\documentclass{amsart}

\usepackage{amsmath}
\usepackage{amsthm}
\usepackage{amssymb}
\usepackage{dsfont}
\usepackage{graphicx}
\usepackage{caption}
\usepackage{subcaption}
\usepackage{color}
\usepackage[english]{babel}
\usepackage{layout}

\newtheorem{theorem}{Theorem}[section]

\theoremstyle{definition}
\newtheorem{definition}[theorem]{Definition}
\newtheorem{remark}{Remark}

\newcommand{\LM}[1]{\hbox{\vrule width.2pt \vbox to#1pt{\vfill \hrule
			width#1pt
			height.2pt}}}
\def\LL{\!
	{\mathchoice{\>\LM7\>}{\>\LM7\>}{\,\LM5\,}{\,\LM{3.35}\,}}}

\def\R{\mathbb{R}}
\def\Z{\mathcal{Z}}
\def\F{\mathcal{F}}

\begin{document}

\author{Matthias Ruf}
\address[Matthias Ruf]{Zentrum Mathematik - M7, Technische Universit\"at M\"unchen, Boltzmannstrasse 3, 85747 Garching, Germany}
\email{mruf@ma.tum.de}

\title{On the continuity of functionals defined on partitions}

\begin{abstract}
We characterize the continuity of prototypical functionals acting on finite Caccioppoli partitions. In the spirit of the classical Reshetnyak continuity theorem for measures that can be used to prove continuity of surface-type functionals defined on single sets of finite perimeter we show that in the multiphase case continuity is equivalent to convergence of the perimeter of the jump set. 
\end{abstract}


\maketitle

\section{Introduction}
In this short note we investigate the continuity of functionals defined on functions of bounded variation taking values in a finite set. More precisely, for an open set $\Omega\subset\R^d$ and $\Z=\{z_1,\dots,z_q\}\subset\R^N$ we consider functionals $F:BV(\Omega,\Z)\to \R$ of the form
\begin{equation}\label{functional}
F(u)=\int_{S_u\cap\Omega}g(x,u^+,u^-,\nu_u)\,\mathrm{d}\mathcal{H}^{d-1}.
\end{equation}
Here $S_u$ denotes the discontinuity set of $u$, $\nu_u=\nu_u(x)$ is the corresponding normal vector at $x\in S_u$ and $u^+,u^-$ are the traces of $u$ on both sides of the discontinuity set. As it is usual in this framework the functional is well-defined if we require the symmetry condition $g(x,a,b,\nu)=g(x,b,a,-\nu)$. Such functionals arise for example in the study of multiphase Cahn-Hilliard fluids \cite{Ba} or the discrete to continuum analysis for spin systems with finitely many ground states \cite{BrCi16}. A general treatment of these functionals from a variational point of view can be found in \cite{AmBrI,AmBrII}. In the recent paper \cite{BrCoGa} the authors proved a density result in the space $BV(\Omega,\mathcal{Z})$ and established continuity of functionals of type (\ref{functional}) along the particular approximating sequence. Here we investigate general continuity properties. We provide a precise characterization of the convergence such that all functionals of the form (\ref{functional}) with $g$ bounded and continuous are themselves continuous with respect to this convergence.  
\\
\hspace*{0.5cm}
As we aim for a rather weak kind of convergence it is convenient to require that the integrand $g$ is bounded and continuous. While continuity in the trace variables is redundant as $\mathcal{Z}$ is a finite set, continuity in $x$ and $\nu$ can surely be dropped if we aim for norm convergence in $BV(\Omega,\Z)$. On the other hand, given a sequence $u_n\in BV(\Omega,\Z)$ such that $u_n\to u$ in $L^1(\Omega)$ we cannot expect that the energy converges as well. In this paper we prove that functionals of the form (\ref{functional}) are continuous along sequences $u_n$ such that $u_n\to u$ in $L^1(\Omega)$ and in addition $\mathcal{H}^{d-1}(S_{u_n}\cap\Omega)\to\mathcal{H}^{d-1}(S_u\cap\Omega)$. This is of course also a necessary condition when we take $g\equiv 1$. To the best of our knowledge such a result is not contained in the mathematical literature.
\\
\hspace*{0.5cm}
This short note is organised as follows: In the first part we give a short introduction to functions of bounded variation. In the second part we prove our main claim. 

\section{Mathematical Preliminaries}\label{sec:fp}
In this section we recall basic facts about functions of bounded variation that can be found in \cite{AFP}.
\begin{definition}\label{bv}
A function $u\in L^1(\Omega)$ is a function of bounded variation, if there exists a finite vector-valued Radon measure $\mu$ on $\Omega$ such that for any $\varphi\in C^{\infty}_c(\Omega,\mathbb{R}^d)$ it holds
\begin{equation*}
\int_{\Omega}u\,{\rm div}\varphi\,\mathrm{d}x=-\int_{\Omega}\langle \varphi,\mu\rangle.
\end{equation*}
In this case we write $u\in BV(\Omega)$ and $Du=\mu$ is the distributional derivative of $u$. A function $u\in L^1(\Omega,\mathbb{R}^N)$ belongs to $BV(\Omega,\mathbb{R}^N)$ if every component belongs to $BV(\Omega)$.  In this case $Du$ denotes the matrix-valued Radon measure consisting of the distributional derivatives of each component.	
\end{definition} 
The spaces $BV_{{\rm loc}}(\Omega)$ and $BV_{{\rm loc}}(\Omega,\mathbb{R}^N)$ are defined as usually. The space $BV(\Omega,\R^N)$ becomes a Banach space when endowed with the norm $\|u\|_{BV(\Omega,\R^N)}=\|u\|_{L^1(\Omega,\R^N)}+|Du|$, where $|Du|$ denotes the total variation of $Du$. When $\Omega$ is a bounded Lipschitz domain, then $BV(\Omega,\R^N)$ is compactly embedded in $L^1(\Omega,\R^N)$. We say that a sequence $u_n$ converges weakly$^*$ in $BV(\Omega,\R^N)$ to $u$ if $u_n\to u$ in $L^1(\Omega,\R^N)$ and $Du_n\overset{*}{\rightharpoonup}Du$ in the sense of measures. We say that $u_n$ converges strictly to $u$ if $u_n\to u$ in $L^1(\Omega,\R^N)$ and $|Du_n|\to |Du|$. Note that strict convergence implies weak$^*$-convergence and that for $\Omega$ with Lipschitz boundary norm-bounded sequences in $BV(\Omega,\R^N)$ are compact with respect to weak$^*$-convergence, but not necessarily with respect to strict convergence.
\\
\hspace*{0.5cm}
We say that a Lebesgue-measurable set $E\subset \mathbb{R}^d$ has finite perimeter in $\Omega$ if its characteristic function $\mathds{1}_E$ belongs to $BV(\Omega)$. We say it has locally finite perimeter in $\Omega$ if $\mathds{1}_E\in BV_{{\rm loc}}(\Omega)$. Let $\Omega^{\prime}$ be the largest open set such that $E$ has locally finite perimeter in $\Omega^{\prime}$. The reduced boundary $\mathcal{F}E$ of $E$ is defined as
\begin{equation*}
\mathcal{F}E:=\left\{x\in \Omega^{\prime}\cap{\rm supp}|D\mathds{1}_E|:\;\nu_E(x)=\lim_{\rho\to 0}\frac{D\mathds{1}_E(B_{\rho}(x))}{|D\mathds{1}_E(B_{\rho}(x))}\text{ exists and }|\nu_E(x)|=1\right\}.
\end{equation*}
Then it holds that $|D\mathds{1}_E|=\mathcal{H}^{d-1}\LL{\mathcal{F}E}$ and $\nu_E$ can be interpreted as a measure theoretic inner normal vector (see also Theorem 3.59 in \cite{AFP}).
\\
\hspace*{0.5cm}
Now we state some fine properties of $BV$-functions. To this end, we need some definitions. A function $u\in L^1(\Omega,\mathbb{R}^N)$ is said to have an approximate limit at $x\in \Omega$ whenever there exists $z\in\mathbb{R}^N$ such that
\begin{equation*}
\lim_{\rho\to 0}\frac{1}{\rho^d}\int_{B_{\rho}(x)}|u(y)-z|\,\mathrm{d}y=0.
\end{equation*}
We let $S_u\subset \Omega$ be the set, where $u$ has no approximate limit. Now we introduce so called approximate jump points. Given $x\in \Omega$ and $\nu\in S^{d-1}$ we set
\begin{equation*}
\begin{cases}
B^+_{\rho}(x,\nu)=\{y\in B_{\rho}(x):\;\langle y-x,\nu\rangle>0\}
\\
B^-_{\rho}(x,\nu)=\{y\in B_{\rho}(x):\;\langle y-x,\nu\rangle<0\}
\end{cases}
\end{equation*}
We say that $x\in \Omega$ is an approximate jump discontinuity of $u$ if there exist $a\neq b\in\mathbb{R}^N$ and $\nu\in S^{d-1}$ such that
\begin{equation*}
\lim_{\rho\to 0}\frac{1}{\rho^d}\int_{B_{\rho}^+(x,\nu)}|u(y)-a|\,\mathrm{d}y=\lim_{\rho\to 0}\frac{1}{\rho^d}\int_{B^-_{\rho}(x,\nu)}|u(y)-b|\,\mathrm{d}y=0.
\end{equation*}
Note that the triplet $(a,b,\nu)$ is determined uniquely up to the change to $(b,a,-\nu)$. We denote it by $(u^+(x),u^-(x),\nu_u(x))$. We let $J_u$ be the set of approximate jump discontinuities of $u$. Then the triplet $(u^+,u^-,\nu_u)$ can be chosen as a Borel function on the Borel set $J_u$. If $u\in BV(\Omega,\mathbb{R}^N)$ it can be shown that $\mathcal{H}^{d-1}(S_u\backslash J_u)=0$. Denoting by $\nabla u$ the density of the absolutely continuous part of $Du$ with respect to the Lebesgue measure, we can decompose the measure $Du$ via
\begin{equation*}
Du(B)=\int_B\nabla u\,\mathrm{d}x+\int_{J_u\cap B}(u^+(x)-u^-(x))\otimes\nu_u(x)\,\mathrm{d}\mathcal{H}^{d-1}+D^cu(B),
\end{equation*}
where $D^cu$ is the so called Cantor part.
\\
\hspace*{0.5cm}
From now on we assume that $\Omega$ is a bounded open set. Given a finite set $\Z=\{z_1,\dots,z_q\}\subset\mathbb{R}^N$ we define the space $BV(\Omega,\Z)$ as the space of those functions $u\in BV(\Omega,\mathbb{R}^N)$ such that $u(x)\in \Z$ almost everywhere. As an immediate consequence of the coarea formula applied to each component, it follows that all level sets $E_i:=\{u=z_i\}$ have finite perimeter in $\Omega$. Moreover, the total variation and the surface measure of $S_u$ are given by
\begin{equation*}
|Du|=\frac{1}{2}\sum_{i=1}^q\sum_{j\neq i}|z_i-z_j|\mathcal{H}^{d-1}(\F E_i\cap\F E_j\cap \Omega),\quad\quad
\mathcal{H}^{d-1}(S_u)=\frac{1}{2}\sum_{i=1}^q\mathcal{H}^{d-1}(\F E_i\cap \Omega).
\end{equation*}

\section{Statement and proof of the main result}
The following Theorem is the main result of this short note. For our proof we use minimal liftings in $BV$ as in \cite{RiSh} (see also \cite{JeJu})
\begin{theorem}\label{uppercont}
Let $u_n,u\in BV(\Omega,\Z)$ be such that $u_n\to u$ in $L^1(\Omega)$ and such that $\mathcal{H}^{d-1}(S_{u_n}\cap\Omega)\to\mathcal{H}^{d-1}(S_u\cap\Omega)$. Let $g:\Omega\times \Z^2\times S^{d-1}\to\R$ be bounded and continuous. Then
\begin{equation*}
\lim_n\int_{S_{u_n}\cap \Omega}g(x,u_n^+,u_n^-,\nu_{u_n})\,\mathrm{d}\mathcal{H}^{d-1}=\int_{S_u\cap \Omega}g(x,u^+,u^-,\nu_u)\,\mathrm{d}\mathcal{H}^{d-1}.
\end{equation*}	
\end{theorem}

\begin{proof}
To reduce notation, we define $F(u)=\int_{S_u\cap \Omega}g(x,u^+,u^-,\nu_u)\,\mathrm{d}\mathcal{H}^{d-1}$. We will just prove upper semicontinuity. The general result then follows applying upper semicontinuity to the functional $-F$. By our assumptions without loss of generality we can assume that $g\geq 0$. For an arbitrary $v\in BV(\Omega,\Z)$ we define for $|Dv|$-almost every $x\in \Omega$ the vector measure $\lambda_x$ via its action on functions $\varphi\in C_0(\mathbb{R}^N)$ by
\begin{equation*}
\int_{\R^N}\varphi(y)\,\mathrm{d}\lambda_x(y)=\frac{\mathrm{d}Dv}{\mathrm{d}|Dv|}(x)\int_0^1\varphi(\theta v^+(x)+(1-\theta)v^-(x))\,\mathrm{d}\theta.
\end{equation*}
To reduce notation, we write $v^{\theta}=\theta v^++(1-\theta)v^-$. Since $v^+,v^-$ are $|Dv|$-measurable, using Fubini's theorem one can show that for any $\varphi\in C_0(\Omega\times\R^N)$ the mapping 
\begin{equation*}
x\mapsto \int_{\R^N}\varphi(x,y)\,\mathrm{d}\lambda_x(y)
\end{equation*}
is $|Dv|$-measurable and essentially bounded. Hence we can define the generalized product $\mu[v]=|Dv|\otimes \lambda_x$ again by its action on $C_0(\Omega\times\mathbb{R}^N)$ setting
\begin{equation*}
\int_{\Omega\times\mathbb{R}^N}\varphi(x,y)\,\mathrm{d}\mu[v](x,y)=\int_{\Omega}\int_{\mathbb{R}^N}\varphi(x,y)\,\mathrm{d}\lambda_x(y)\,\mathrm{d}|Dv|(x);
\end{equation*}
see also Definition 2.27 in \cite{AFP}. We next claim that up to a negligible set it holds that 
\begin{equation}\label{polar}
\frac{\mathrm{d}\mu[v]}{\mathrm{d}|\mu[v]|}(x,y)=\frac{\mathrm{d}Dv}{\mathrm{d}|Dv|}(x).
\end{equation}
Indeed, Corollary 2.29 in \cite{AFP} yields $|\mu[v]|=|Dv|\otimes |\lambda_x|$. As the defining formula for the generalized product extends to integrable functions, we infer that
\begin{align*}
\int_{\Omega\times\mathbb{R}^N}&\varphi(x,y)\frac{\mathrm{d}Dv}{\mathrm{d}|Dv|}(x)\,\mathrm{d}|\mu[v]|(x,y)=\int_{\Omega}\int_{\mathbb{R}^N}\varphi(x,y)\frac{\mathrm{d}Dv}{\mathrm{d}|Dv|}(x)\,\mathrm{d}|\lambda_x|(y)\,\mathrm{d}|Dv|(x)
\\
&=\int_{\Omega}\int_{\mathbb{R}^N}\varphi(x,y)\,\mathrm{d}\lambda_x(y)\,\mathrm{d}|Dv|(x)=\int_{\Omega\times \mathbb{R}^N}\varphi(x,y)\,\mathrm{d}\mu[v](x,y),
\end{align*}
where we have used that $\lambda_x=\frac{\mathrm{d}Dv}{\mathrm{d}|Dv|}(x)|\lambda_x|$. Hence (\ref{polar}) follows by uniqueness of the polar decomposition of measures. Because of (\ref{polar}) and the generalized product structure of $|\mu[v]|$, by an approximation argument it holds that
\begin{align}\label{toreshetnyak}
\int_{\Omega\times\mathbb{R}^N}f(x,y,\frac{\mathrm{d}\mu[v]}{\mathrm{d}|\mu[v]|}(x,y))\,\mathrm{d}|\mu[v]|(x,y)
&=\int_{\Omega}\int_{\mathbb{R^N}}f(x,y,\frac{\mathrm{d}Dv}{\mathrm{d}|Dv|}(x))\,\mathrm{d}|\lambda_x|(y)\,\mathrm{d}|Dv|(x)\nonumber
\\
&=\int_{\Omega}\int_0^1f(x,v^{\theta},\frac{\mathrm{d}Dv}{\mathrm{d}|Dv|}(x))\,\mathrm{d}\theta\,\mathrm{d}|Dv|(x)
\end{align}
for every nonnegative function $f\in C(\Omega\times\mathbb{R}^N\times S^{N\times d-1})$. In \cite{JeJu} it was proven that if $v_n\to v$ strictly in $BV(\Omega,\mathbb{R}^N)$, then $\mu[v_n]\overset{*}{\rightharpoonup} \mu[v]$ and $|\mu[v_n](\Omega\times\mathbb{R}^N)|\to|\mu[v](\Omega\times\mathbb{R}^N)|$. The idea now is to apply the classical Reshetnyak continuity theorem (see for instance \cite{Res,Sp}) with an appropriate $f$ and a strictly converging sequence. To this end we transform the set $\Z$ so that averages of the jump functions $u^{\pm}$ encode the values of the traces and such that the convergence assumptions yield strict convergence. Recall that $q=\#\Z$. We define the mapping $T:\Z\to\mathbb{R}^{q}$ via $T(z_i)=e_i$. Next we construct the function $f$. Given $i<j$ we consider the set 
\begin{equation*}
L_{ij}=\{\lambda T(z_i)+(1-\lambda)T(z_j):\;\lambda\in (1/4,3/4)\}.
\end{equation*}
Observe that by construction of the set $T(\Z)$ it holds $L_{ij}\cap L_{kl}=\emptyset$ whenever $\{i,j\}\neq\{k,l\}$. Given $\delta>0$ we next choose a cut-off function $\theta^{\delta}_{ij}: [T(z_i),T(z_j)]\to [0,1]$ such that $\theta^{\delta}_{ij}=1$ on $L_{ij}$ and $\theta^{\delta}_{ij}(x)=0$ if ${\rm dist}(x,L_{ij})\geq\delta$. Set $f_{\delta}\in C(D\times\mathbb{R}^{q}\times S^{q\times d-1})$ as any continuous nonnegative extension of the function
\begin{equation*}
f_{\delta}(x,u,\xi)=
\frac{\theta_{ij}^{\delta}(u)}{\sqrt{2}\mathcal{H}^1(L_{ij})}g(x,z_i,z_j,\frac{\xi^Te_1}{|\xi^Te_1|})|\xi^Te_1|\quad\text{ if }u\in [T(z_i),T(z_j)].
\end{equation*}
First observe that this well-defined due to the ordering $i<j$ (also in the case $\xi^Te_1=0$ as $g$ is bounded). Moreover, for $\delta$ small enough such an extension exists by the properties of the cut-off function. Now for any $T(u)\in BV(\Omega,T(\Z))$, with a suitable orientation of the normal vector, for $|DT(u)|$-almost every $x\in \Omega$ it holds that
\begin{equation*}
\begin{split}
&\frac{\mathrm{d}DT(u)}{\mathrm{d}|DT(u)|}(x)=\frac{1}{\sqrt{2}}\sum_{i<j}(T(z_i)-T(z_j))\otimes\nu_u(x)\mathds{1}_{\F E_i\cap\F E_j}(x),
\\
&|DT(u)|=\sqrt{2}\sum_{i<j}\mathcal{H}^{d-1}\LL{(\F E_i\cap\F E_j)},
\end{split}
\end{equation*} 
where $E_i=\{u=z_i\}$. Therefore we can rewrite with a nonnegative error $\mathcal{O}(\delta)$
\begin{align*}
\int_{\Omega}\int_{0}^1&f_{\delta}(x,T(u)^{\theta},\frac{\mathrm{d}DT(u)}{\mathrm{d}|DT(u)|}(x))\,\mathrm{d}\theta\,\mathrm{d}|DT(u)|(x)
\\
&=\int_{\Omega}\sum_{i<j}g(x,z_i,z_j,\nu_u)\,\mathrm{d}\mathcal{H}^{d-1}\LL{(\F E_i\cap\F E_j)}+\mathcal{O}(\delta)\mathcal{H}^{d-1}(S_u\cap \Omega)
=F(u)+\mathcal{O}(\delta)\mathcal{H}^{d-1}(S_u\cap \Omega).
\end{align*}
If $u_n,u$ are as in the claim, then $T(u_n)\to T(u)$ in $L^1(\Omega)$ and moreover $|DT(u_n)|=\sqrt{2}\mathcal{H}^{d-1}(S_{u_n}\cap \Omega)\to\sqrt{2}\mathcal{H}^{d-1}(S_u\cap \Omega)=|DT(u)|$, so that $T(u_n)$ converges strictly to $T(u)$. Hence we conclude from (\ref{toreshetnyak}) and the classical Reshetnyak continuity theorem applied to the measures $\mu[T(u_n)],\mu[T(u)]$ that
\begin{align*}
\limsup_nF(u_n)&\leq\lim_n \int_{\Omega}\int_{0}^1f_{\delta}(x,T(u_n)^{\theta},\frac{\mathrm{d}DT(u_n)}{\mathrm{d}|DT(u_n)|}(x))\,\mathrm{d}\theta\,\mathrm{d}|DT(u_n)|(x)
\\
&=\int_{\Omega}\int_{0}^1f_{\delta}(x,T(u)^{\theta},\frac{\mathrm{d}DT(u)}{\mathrm{d}|DT(u)|}(x))\,\mathrm{d}\theta\,\mathrm{d}|DT(u)|(x)
\leq F(u)+\mathcal{O}(\delta)\mathcal{H}^{d-1}(S_u\cap \Omega).
\end{align*}
The claim follows by the arbitrariness of $\delta$.
\end{proof}
  
\begin{remark}\label{consequences}
Taking $g(x,u^+,u^-,\nu)=|u^+-u^-|$, Theorem \ref{uppercont} yields that $L^1(\Omega)$-convergence combined with the convergence of $\mathcal{H}^{d-1}(S_u)$ implies strict convergence in $BV(\Omega,\mathcal{Z})$. The converse is false in general as can be seen already by the following one-dimensional example. Given $n\in\mathbb{N}$ we set $u_n:(-1,1)\to\R$ as $u_n(x)=\mathds{1}_{(1/n,2/n)}+2\mathds{1}_{(2/n,1)}$. Then $u_n$ converges strictly to the function $u=2\mathds{1}_{(0,1)}$ while $\mathcal{H}^0(S_{u_n})=2$ but $\mathcal{H}^0(S_u)=1$.
\end{remark}


\end{document}